\documentclass[a4paper,UKenglish,cleveref, autoref, thm-restate]{lipics-v2021}

\pdfoutput=1 
\hideLIPIcs  

\usepackage{tikz-cd}
\usepackage{mathtools}

\newcommand{\op}{\textup{op}}

\newcommand{\id}{\textup{id}}
\newcommand{\ob}{\textup{ob}}

\newcommand{\sk}{\textup{sk}}
\newcommand{\cosk}{\textup{cosk}}
\newcommand{\tr}{\textup{tr}}
\newcommand{\ho}{\textup{ho}}
\newcommand{\nerve}{\textup{nerve}}

\newcommand{\cast}{\textup{cast}}

\newcommand{\lan}{\textup{lan}}

\newcommand{\Hom}{\textup{Hom}}

\newcommand{\cat}[1]{\textup{\textsf{#1}}}

\newcommand{\NN}{\mathbb{N}}

\DeclareMathAlphabet{\mathbbe}{U}{bbold}{m}{n}

\newcommand{\cA}{\mathcal{A}}

\newcommand{\cE}{\mathcal{E}}

\newcommand{\DDelta}{\mathbbe{\Delta}}

\newcommand{\Fin}{\cat{Fin}}
\newcommand{\Cat}{\mathcal{C}\cat{at}}
\newcommand{\Set}{\mathcal{S}\cat{et}}
\newcommand{\Quiv}{\mathcal{Q}\cat{uiv}}
\newcommand{\QCat}{\mathcal{QC}\cat{at}}
\newcommand{\rQuiv}{\cat{r}\mathcal{Q}\cat{uiv}}

\newcommand{\rightharpoons}{\mathrel{\mathrlap{\raisebox{0.7pt}{\ensuremath{\rightharpoonup}}}{\raisebox{-0.7pt}{\ensuremath{\rightharpoondown}}}}}

\font\maljapanese=dmjhira at 2ex 
\def\yo{\textrm{\maljapanese\char"48}}

\usepackage{xcolor}
\definecolor{keywordcolor}{rgb}{0.7, 0.1, 0.1}   
\definecolor{tacticcolor}{rgb}{0.0, 0.1, 0.6}    
\definecolor{commentcolor}{rgb}{0.4, 0.4, 0.4}   
\definecolor{symbolcolor}{rgb}{0.0, 0.1, 0.6}    
\definecolor{sortcolor}{rgb}{0.1, 0.5, 0.1}      
\definecolor{attributecolor}{rgb}{0.7, 0.1, 0.1} 

\usepackage{listings}

\newcommand{\fixme}[1]{}

\newcommand{\liboneLab}{\href{https://1lab.dev/}{\textsf{1Lab}}}

\newcommand{\libmathlib}{\href{https://github.com/leanprover-community/mathlib}{\textsf{mathlib}}}
\newcommand{\libunimath}{\href{https://github.com/UniMath/UniMath}{\textsf{UniMath}}}

\newcommand{\ldoc}[2][]{\href{https://leanprover-community.github.io/mathlib4_docs/find/?pattern=#1#2\#doc}{\texttt{#2}}}
\newcommand{\cdoc}[2][]{\href{https://leanprover-community.github.io/mathlib4_docs/find/?pattern=CategoryTheory.#1#2\#doc}{\texttt{#2}}}
\newcommand{\ldocTwo}[3][]{\href{https://leanprover-community.github.io/mathlib4_docs/find/?pattern=#1#2\%E2\%82\%82#3\#doc}{\texttt{#2${}_2$#3}}}
\newcommand{\cdocTwo}[3][]{\href{https://leanprover-community.github.io/mathlib4_docs/find/?pattern=CategoryTheory.#1#2\%E2\%82\%82#3\#doc}{\texttt{#2${}_2$#3}}}
\newcommand{\ldocTwoTwo}[4][]{\href{https://leanprover-community.github.io/mathlib4_docs/find/?pattern=#1#2\%E2\%82\%82#3\%E2\%82\%82#4\#doc}{\texttt{#2${}_2$#3${}_2$#4}}}

\makeatletter
\@fleqnfalse
\@mathmargin\@centering

\let\c@equation\c@theorem
\makeatother

\title{Formalizing colimits in $\Cat$}

\author{Mario Carneiro}{Chalmers University of Technology, Sweden} {marioc@chalmers.se}{https://orcid.org/0000-0002-0470-5249}{}

\author{Emily Riehl\footnote{Corresponding author}}{Department of Mathematics, Johns Hopkins University, 3400 N Charles Street, Baltimore, MD, USA }{eriehl@jhu.edu}{https://orcid.org/0000-0002-8465-8859}{NSF DMS-2204304, AFOSR FA9550-21-1-0009, ARO W911NF-20-1-0082}

\authorrunning{M. Carneiro and E. Riehl} 

\Copyright{Mario Carneiro and Emily Riehl} 



\category{} 

\relatedversion{} 

\supplement{}

\supplementdetails[subcategory={}, cite={}, swhid={}]{Software}{https://github.com/leanprover-community/mathlib4/tree/cdf36433a}

\acknowledgements{This project was first instantiated at the 2024 Prospects of Formal Mathematics Trimester Program at the Hausdorff Institute for Mathematics in Bonn. The authors are grateful to the organizers of that program for facilitating many new collaborations and to the HIM for its hospitality. They also wish to thank Jo\"{e}l Riou for the excellent suggestions he made during the Mathlib PR review process. The final journal version contains myriad expository improvements suggested by the reviewers.}

\nolinenumbers 

\EventEditors{Yannick Forster and Chantal Keller}
\EventNoEds{2}
\EventLongTitle{16th International Conference on Interactive Theorem Proving (ITP 2025)}
\EventShortTitle{ITP 2025}
\EventAcronym{ITP}
\EventYear{2025}
\EventDate{September 29--October 2, 2025}
\EventLocation{Reykjavik, Iceland}
\EventLogo{}
\SeriesVolume{352}
\ArticleNo{20}

\begin{document}

\maketitle

\begin{abstract}
  Certain results involving ``higher structures'' are not currently accessible to computer formalization because the prerequisite $\infty${\nobreakdash-}category theory has not been formalized. To support future work on formalizing $\infty${\nobreakdash-}category theory in Lean's mathematics library, we formalize some fundamental constructions involving the 1-category of categories. Specifically, we construct the left adjoint to the nerve embedding of categories into simplicial sets, defining the homotopy category functor. We prove further that this adjunction is reflective, which allows us to conclude that $\Cat$ has colimits. To our knowledge this is the first formalized proof that the nerve functor is a fully faithful right adjoint and that the category of categories is cocomplete.

\end{abstract}

\maketitle


\section{Introduction}

One of the myriad ways in which ``category theory eats itself'' is that the categorically-defined notions of limits and colimits can be interpreted in the category $\Cat$ itself, whose objects are categories\footnote{An important point that is often elided on paper but is necessary to treat precisely in formalization is that the category of categories \cdoc{Cat}.\lstinline|\{v,u\}| is parametrized by two implicit universe levels, one \lstinline|u| for the type of objects and one \lstinline|v| for the hom-types.} and functors between them. Any suitably-small diagram of categories and functors has a \emph{limit} in $\Cat$ which is constructed by forming the limit of the corresponding diagrams of objects and hom-sets.\footnote{
Specifically, the category \cdoc{Cat}.\lstinline|\{u,u\}| of small categories at any universe level \lstinline|u|, has all limits indexed by categories $J$ in \cdoc{Cat}.\lstinline|\{u,u\}|.  The definition of \emph{limit} is reviewed in \S\ref{sec:strict-segal}; \emph{colimits} are dual, formed by reversing the arrows.} For example, the \emph{product} of two categories $C$ and $D$ is the category whose objects are defined as the product $\ob\,\,C \times \ob\,D$ of the types of objects of $C$ and $D$. Given objects $c,c' : \ob\,C$ and $d,d' : \ob\,D$, the hom-type from $(c,d)$ to $(c',d')$ in the product category is the product $\Hom_C(c,c') \times \Hom_D(d,d')$ of the hom-types in $C$ and $D$.
The generalization of this construction to arbitrary diagram shapes is formalized in Lean's \libmathlib{} as \cdoc{Cat.instHasLimits} by Kim Morrison \cite{Morrison}.

It is also the case that any suitably-small diagram in $\Cat$ has a colimit, but the construction is considerably more complicated. To illustrate, consider the terminal category $\Fin(1)$ (with one object and only its identity arrow) and the walking arrow $\Fin(2)$ (with two objects $0$ and $1$ and a single non-identity morphism $\ell \colon 0 \to 1$). The \emph{coequalizer} of the parallel pair of functors $0,1 \colon \Fin(1) \to \Fin(2)$ that pick out these two objects is the category $\cat{B}\NN$ with a single object but with a morphism for each natural number, with composition defined by addition and the morphism $0$ serving as the identity. The quotient functor $q \colon \Fin(2) \to \cat{B}\NN$ sends the morphism $\ell$ to the generating non-identity morphism $1$.
\begin{equation}\label{eq:first-example} \begin{tikzcd} \Fin(1) \arrow[r, shift left=.4em, "0"] \arrow[r, shift right=.4em, "1"'] & \Fin(2) \arrow[r, dashed, "q"] & \cat{B}\NN \end{tikzcd}  \qquad  \begin{tikzcd}[row sep=tiny, column sep=large] & \bullet \arrow[dd, "\ell"] \arrow[dr, mapsto, dashed]\\ \bullet \arrow[ur, mapsto, "0"] \arrow[dr, mapsto, "1"'] & & \bullet  \arrow[loop above, "1"] \arrow[loop below, "\cdots"] \arrow[loop right, "2"]  \\ & \bullet \arrow[ur, mapsto, dashed]\end{tikzcd} \end{equation}

Given the much greater complexity involved in the construction of colimits in $\Cat$, it is perhaps unsurprising that none of the libraries of formal mathematics that currently exist in Agda \cite{hott-in:agda, HuCarette2021, 1Lab, agda-unimath, VezzosiMoertbergAbel2021}, HOL Light \cite{Harrison2009}, Isabelle \cite{CZH_Foundations-AFP, CZH_Elementary_Categories-AFP, CZH_Universal_Constructions-AFP, NipkoWenzelPaulson2002}, Lean \cite{mathlib2020}, or Rocq \cite{BauerGrossLumsdaineShulmanSozeauSpitters2017,GrossChlipalaSpivak, UniMath} contain a proof that $\Cat$ is cocomplete. In particular, we discovered that this result was missing from Lean's \libmathlib{} while attempting to formalize a theorem that is relevant to $\infty${\nobreakdash-}\emph{category theory}, the extension of ordinary category theory to settings where there are ``higher homotopies'' between maps.\footnote{As noted recently by the authors of \cite{KRW}, ``myriad recent results from algebraic K-theory \cite{BlumbergGepnerTabuada2013}, derived and spectral algebraic geometry \cite{Lurie2004,Lurie:2018sag}, the Langlands program \cite{FarguesScholze2021}, and symplectic geometry \cite{NadlerTanaka2020} are inaccessible to formalization'' because \libmathlib{}  contains relatively few results from $\infty${\nobreakdash-}\emph{category theory}.} When this project commenced in Summer 2024, \libmathlib{} contained an important 1-categorical pre-requisite to $\infty${\nobreakdash-}category theory, namely the construction of the \emph{nerve} of a 1-category, formalized by Jo\"{e}l Riou. 

The nerve of a category~$C$, is a \emph{simplicial set}, a family of types $(\nerve\;C_n)_{n : \NN}$ equipped with various maps between them. More precisely, a \emph{simplicial set} is a contravariant functor indexed by the \ldoc{SimplexCategory} $\DDelta$; objects are natural numbers $n : \NN$ commonly denoted by $[n]$ and morphisms $[m] \to [n]$ are defined in \ldoc{SimplexCategory.Hom} to be order-preserving functions $\Fin(m+1) \to \Fin(n+1)$. The category \ldoc{SSet} of simplicial sets is the functor category $\Set^{\DDelta^\op}$. Explicitly, $\nerve\ C_n$ is the type of functors $\Fin(n+1)\to C$, where $\Fin(n+1)$ is the ordinal category with $n+1$ objects depicted below:
\[ \begin{tikzcd} 0 \arrow[r] & 1 \arrow[r]  & \cdots \arrow[r] & n-1 \arrow[r] &  n \end{tikzcd} \]
\begin{lstlisting}
/-- `ComposableArrows C n` is the type of functors `Fin (n + 1) ⥤ C`. -/
abbrev !\cdoc{ComposableArrows}! (n : ℕ) := !\ldoc{Fin}! (n + 1) ⥤ C

/-- The nerve of a category -/
def !\cdoc{nerve}! (C : Type) [Category C] : !\ldoc{SSet}! where
  obj Δ := !\cdoc{ComposableArrows}! C (Δ.unop.len)
  map f x := x.!\cdoc[ComposableArrows.]{whiskerLeft}! (!\ldoc{SimplexCategory.toCat}!.map f.unop)
\end{lstlisting}

Up to equivalence, $\nerve\ C_0$ is the type of objects of $C$, $\nerve\ C_1$ is the type of arrows of $C$, $\nerve\ C_2$ is the type of composable pairs of arrows, and so on. In the presence of the ``higher homotopies'' alluded to above, this data is used to index ``homotopy coherent diagrams'' of shape $C$, incarnating the 1-category $C$ as an $\infty${\nobreakdash-}category.

The nerve construction defines a functor $\nerve \colon \Cat \to \Set^{\DDelta^\op}\!$ valued in the category of simplicial sets. Our objective was to define its left adjoint $\ho \colon \Set^{\DDelta^\op}\!\! \to \Cat$, a functor which takes a simplicial set to its \emph{homotopy category}. This functor is used to collapse an $\infty${\nobreakdash-}category to its quotient 1-category and will also be used to define the 2-category of $\infty${\nobreakdash-}categories described in \S\ref{sec:future}. As noted for instance in \cite[6.5.iv]{Riehl:2016cc}, this adjunction exists for formal reasons as a special case of a result which has already been formalized in \libmathlib{}. Bhavik Mehta and Jo\"{e}l Riou formalized the construction of a general family of ``Yoneda adjunctions'':\footnote{We've simplified the statement here for readability. The only hypothesis that is actually required is that the left Kan extension of the functor $A$ along the Yoneda embedding $\yo \colon \cA \to \Set^{\cA^\op}$ exists. In practice, one proves that this left Kan extension exists by showing that $\cA$ is small and $\cE$ is cocomplete.}

\begin{theorem}[\cdoc{Presheaf.yonedaAdjunction}] Suppose $A \colon \cA \to \cE$ is any functor,  where $\cA$ is small and $\cE$ is cocomplete.
 Then the left Kan extension of $A$ along the Yoneda embedding $\yo \colon \cA \to \Set^{\cA^\op}$ is left adjoint to the restricted Yoneda functor:
  $ \begin{tikzcd}[column sep=large] \cE \arrow[r, bend right=20, "{\cE(A-,-)}"' pos=.6] \arrow[r, phantom, "\bot"] & \Set^{\cA^\op}\!\!. \arrow[l, bend right=20, "{\lan_{\yo}A}"' pos=.4]
  \end{tikzcd}$
\end{theorem}

The nerve adjunction is an example of a Yoneda adjunction defined relative to the functor $\ldoc{SimplexCategory.toCat} \colon \DDelta \to \Cat$ which sends $n : \NN$ to the ordinal category $\Fin(n+1)$. However, we could not apply this result because \libmathlib{} did not contain a proof that $\Cat$ has all colimits, only a proof that $\Cat$ has limits, which are far easier to construct.  Consulting the literature again \cite[4.5.16]{Riehl:2016cc}, we were reminded that colimits in $\Cat$ can be constructed as a corollary of the adjunction we were hoping to define: since the nerve functor is fully faithful, the adjunction exhibits $\Cat$ as a reflective subcategory of a cocomplete category. Thus our new formalization objective was to give a direct proof of the following pair of results:

\begin{theorem}[\cdoc{nerveAdjunction}]\label{thm:nerveAdj} The nerve functor admits a left adjoint defined by the functor that sends a simplicial set to its homotopy category:
$ \begin{tikzcd}[column sep=large] \Cat \arrow[r, bend right=20, "{\nerve}"' pos=.6] \arrow[r, phantom, "\bot"] & \Set^{\DDelta^\op}\!. \arrow[l, bend right=20, dashed, "{\ho}"' pos=.45]
\end{tikzcd}$
\end{theorem}

\begin{theorem}[\cdoc{nerveFunctor.fullyfaithful}]\label{thm:nerveFF} The nerve functor is full and faithful.
\end{theorem}

A functor is \emph{reflective} if it is a fully faithful right adjoint.
\begin{lstlisting}
/-- A functor is !{\color{commentcolor}\emph{*reflective*}}!, or !{\color{commentcolor}\emph{*a~reflective~inclusion*}}!, if it is fully faithful and right adjoint. -/
class !\cdoc{Reflective}! (R : D ⥤ C) extends R.!\cdoc[Functor.]{Full}!, R.!\cdoc[Functor.]{Faithful}! where
  /-- a choice of a left adjoint to `R' -/
  L : C ⥤ D
  /-- `R' is a right adjoint -/
  adj : L ⊣ R
\end{lstlisting}
As formalized by Kim Morrison, Jack McKoen, and Bhavik Mehta, any reflective subcategory of a cocomplete category has colimits.

\begin{lstlisting}
/-- If `C' has colimits then any reflective subcategory has colimits. -/
theorem !\cdoc{hasColimits\_of\_reflective}! (R : D ⥤ C) [!\cdoc{Reflective}! R]
    [!\cdoc[Limits.]{HasColimits}! C] : !\cdoc[Limits.]{HasColimits}! D := ...
\end{lstlisting}

Thus, as an immediate corollary of Theorems \ref{thm:nerveAdj} and \ref{thm:nerveFF}:\footnote{Most of the Lean statements here have been simplified to avoid asides about universes, but in this theorem statement it is important that we only assert the existence of \lstinline|u|-\emph{small} colimits in the category of \lstinline|u|-\emph{small} categories, which is indicated by the use of universe \lstinline|u| in both objects and homs.}
\begin{lstlisting}
/-- The category of small categories has all small colimits as a reflective
  subcategory of the category of simplicial sets, which has all small colimits. -/
instance : !\cdoc[Limits.]{HasColimits}! !\cdoc{Cat}!.{u, u} := !\cdoc{hasColimits\_of\_reflective}! !\cdoc{nerveFunctor}!
\end{lstlisting}

\begin{corollary}[\cdoc{Cat.instHasColimits}]\label{cor:cat-colimits} $\Cat$ is cocomplete.
\end{corollary}
\begin{proof}
The presheaf category $\Set^{\DDelta^\op}$ is cocomplete, with colimits defined objectwise in $\Set$. On account of the adjunction with a fully faithful right adjoint
$ \begin{tikzcd}[column sep=large] \Cat \arrow[r, bend right=20, hook, "{\nerve}"' pos=.6] \arrow[r, phantom, "\bot"] & \Set^{\DDelta^\op}\! ,\!\! \arrow[l, bend right=20, "{\ho}"' pos=.45]
\end{tikzcd}$  the reflective subcategory $\Cat$ inherits these colimits, defined by taking the nerves of the categories in the original diagram, forming the colimit in $\Set^{\DDelta^\op}\!\!$, and then applying the homotopy category functor to the colimit cone.
\end{proof}

Importantly for the aim of contributing to a general-purpose library, the construction given in the proof of Corollary \ref{cor:cat-colimits} results in a good description of colimits in $\Cat$. Applied to the example \eqref{eq:first-example}, Corollary \ref{cor:cat-colimits} tells us to first form the coequalizer in simplicial sets
\[ \begin{tikzcd} \Delta^0 \arrow[r, shift left=.4em, "0"] \arrow[r, shift right=.4em, "1"'] & \Delta^1 \arrow[r, dashed, "q"] & S^1 \end{tikzcd}\] and then take the homotopy category. As described in Definition \ref{defn:homotopy-category} below, the homotopy category of $S^1$ is the free category defined by the underlying reflexive quiver --- which has a single object and a single non-identity endo-arrow --- modulo relations witnessed by non-degenerate 2-simplices. As the simplicial set $S^1$ has no non-degenerate 2-simplices, $\ho(S^1) \cong \cat{B}\NN$ is the free category on a single non-identity endo-arrow, which is the category $\cat{B}\NN$ described above. The construction for other diagram shapes is similar.

Other methods for proving that $\Cat$ is cocomplete are discussed in \S\ref{sec:related-work}, on related work. These methods typically invoke the theorem that a colimit of a diagram of any shape can be expressed as a coequalizer between two coproducts \cite[V.2.2]{MacLane:1998cw}, so that it suffices to give explicit constructions of coproducts and coequalizers in $\Cat$. This works, but it results in more convoluted expressions for concrete colimits; for our running example, it would calculate the colimit $\cat{B}\NN$ as the coequalizer of a pair of functors between the coproducts
\[ \begin{tikzcd} \Fin(1) \sqcup \Fin(1) \sqcup \Fin(2) \arrow[r, shift left=.4em] \arrow[r, shift right=.4em] & \Fin(1) \sqcup \Fin(2) \arrow[r, dashed] & \cat{B}\NN \end{tikzcd}\]
which is more complicated than it needs to be.

In \S\ref{sec:formalization}, we describe our formalizations of Theorems \ref{thm:nerveAdj} and \ref{thm:nerveFF}. After a PR process that took five and a half months, this entire formalization is now a part of \libmathlib{}. In \S\ref{sec:challenges}, we describe some challenges we encountered along the way. In \S\ref{sec:related-work}, we discuss alternate theoretical approaches to the construction of colimits in $\Cat$ and related formalization projects.  In \S\ref{sec:future}, we discuss plans for future work that will expand and apply these results.

\section{The formalization}\label{sec:formalization}

The formalized proof of our first main result, Theorem \ref{thm:nerveAdj}, which defines the left adjoint to the nerve functor, takes one line:
\begin{lstlisting}
def !\cdoc{nerveAdjunction}! : !\ldoc[SSet.]{hoFunctor}! ⊣ !\cdoc{nerveFunctor}! :=
  !\cdoc{Adjunction.ofNatIsoRight}! ((!\ldoc{SSet.coskAdj}! 2).comp !\cdocTwo{nerve}{Adj}!) !\cdocTwo{Nerve.cosk}{Iso}!.symm
\end{lstlisting}
Unpacking, this tells us that up to a natural isomorphism \cdocTwo{Nerve.cosk}{Iso} between right adjoints abbreviated as ``$\upsilon$'' in the diagram below, the nerve adjunction is defined as a composite of two other adjunctions:
\[ \begin{tikzcd} \Cat \arrow[r, bend right=20, dashed, "\nerve_2"' pos=.55] \arrow[rr, bend right=47, "\nerve"', "{\cong\Uparrow\upsilon}"]\arrow[r, phantom, "\bot"] & \Set^{\DDelta_{\leq 2}} \arrow[l, bend right=20, dashed,  "\ho_2"' pos=.4] \arrow[r, bend right=20, "\cosk_2"' pos=.45] \arrow[r, phantom, "\bot"] & \Set^{\DDelta^\op} \arrow[l, bend right=20, "\tr_2"'] \arrow[ll, bend right=47, "\ho"', "{\rotatebox{270}{$\coloneqq$}}"] \end{tikzcd}\]
In particular, the homotopy category functor is defined as the composite of the two left adjoints:
\begin{lstlisting}
/-- The functor that takes a simplicial set to its homotopy category by passing through the 2-truncation. -/
def !\ldoc[SSet.]{hoFunctor}! : !\ldoc{SSet}! ⥤ !\cdoc{Cat}! := !\ldoc{SSet.truncation}! 2 ⋙ !\ldocTwo[SSet.]{Truncated.hoFunctor}{}
\end{lstlisting}

One of these adjunctions, \texttt{\ldoc{SSet.coskAdj} 2}, had already been formalized, so our task was to define the other adjunction \cdocTwo{nerve}{Adj} and construct the isomorphism \cdocTwo{Nerve.cosk}{Iso}. We first describe the construction of the natural isomorphism \cdocTwo{Nerve.cosk}{Iso}, which exhibits an important property of the nerve functor called ``2-coskeletality.'' Then, we  describe our formalization of the adjunction \cdocTwo{nerve}{Adj} and our proof that its right adjoint is fully faithful, which implies that the nerve functor is fully faithful as well.

\subsection{The 2-coskeletality of the nerve}

The fact that the ``homotopy category $\dashv$ nerve'' adjunction can be constructed as a composite of two other adjunctions follows from the fact that the nerve functor factors (up to isomorphism) as a composite of two right adjoints, through a category that we now introduce. Recall that the category of simplicial sets is defined as a presheaf category $\Set^{\DDelta^\op}\!\!$. Given $n : \NN$, we define the $n$-\emph{truncated simplex category} $\DDelta_{\leq n}$ as the full subcategory $\DDelta_{\leq n} \subset \DDelta$ spanned by the objects $[0], \ldots, [n]$. The category of $n$-\emph{truncated simplicial sets}, $\Set^{\DDelta_{\leq n}^\op}$, is then the category of presheaves over $\DDelta_{\leq n}$. The restriction functor, known as $n$-\emph{truncation}, has fully faithful left and right adjoints defined by left and right Kan extension:\footnote{The right and left Kan extensions exist because $\DDelta$ is small and $\Set$ is complete and cocomplete.}
\[ \begin{tikzcd} \Set^{\DDelta_{\leq n}^\op}\arrow[r, bend left, hook', "\sk_n", "\bot"']\arrow[from=r, "\tr_n" description]  \arrow[r, bend right, hook, "\cosk_n"', "\bot"]& \Set^{\DDelta^\op}  \end{tikzcd}\]
\begin{lstlisting}
/-- The adjunction between the n-skeleton and n-truncation. -/
def !\cdoc[SimplicialObject.]{skAdj}! : !\cdoc[SimplicialObject.]{Truncated.sk}! (C := C) n ⊣ !\cdoc[SimplicialObject.]{truncation}! n := !\cdoc[Functor.]{lanAdjunction}! _ _

/-- The adjunction between n-truncation and the n-coskeleton. -/
def !\cdoc[SimplicialObject.]{coskAdj}! : !\cdoc[SimplicialObject.]{truncation}! (C := C) n ⊣ !\cdoc[SimplicialObject.]{Truncated.cosk}! n := !\cdoc[Functor.]{ranAdjunction}! _ _
\end{lstlisting}

For any functor valued in simplicial sets, we can obtain its $n$-truncated version by postcomposing with the functor $\tr_n \colon \Set^{\DDelta^\op} \to \Set^{\DDelta_{\leq n}^\op}$. In particular, the 2-truncated nerve $\nerve_2$ is defined to be the composite functor:
\[ \begin{tikzcd} \nerve_2 \colon \Cat \arrow[r, "\nerve"] & \Set^{\DDelta^\op\!\!} \arrow[r, "\tr_2"] & \Set^{\DDelta_{\leq 2}^\op}\,. \end{tikzcd}\]
\begin{lstlisting}
/-- The essential data of the nerve functor is contained in the 2-truncation, which is recorded by the composite functor `nerveFunctor₂'. -/
def !\cdocTwo[Nerve.]{nerveFunctor}{}! : !\cdoc{Cat}! ⥤ !\ldoc{SSet.Truncated}! 2 := !\cdoc{nerveFunctor}! ⋙ !\ldoc[SSet.]{truncation}! 2
\end{lstlisting}

The canonical natural transformation $\upsilon \colon \nerve \Rightarrow \cosk_2\,\nerve_2$ is defined by left whiskering the unit $\eta \colon \id \Rightarrow \cosk_2\, \tr_2$  of the adjunction $\tr_2 \dashv \cosk_2$ with the nerve functor:
\[ \begin{tikzcd}[row sep=tiny] & \Set^{\DDelta_{\leq 2}^\op} \arrow[dr, "\cosk_2"] && [-1.5em] & [-1.5em] & & \Set^{\DDelta_{\leq 2}^\op} \arrow[dr, "\cosk_2"] \\ \Cat \arrow[rr, bend right=15, "\nerve"']  \arrow[ur, "\nerve_2"] \arrow[rr, phantom, "\Uparrow\upsilon"] & & \Set^{\DDelta^\op} &  \coloneqq & \Cat \arrow[r, "\nerve"'] \arrow[urr, bend left=15, "\nerve_2", "{\rotatebox{300}{$\coloneqq$}}"'] & \Set^{\DDelta^\op} \arrow[ur, "\tr_2"] \arrow[rr, bend right=15, equals] \arrow[rr, phantom, "\Uparrow\eta"] & &  \Set^{\DDelta^\op}
\end{tikzcd} \]
The natural transformation $\upsilon$ is the underlying map of the isomorphism \cdocTwo{Nerve.cosk}{Iso}.
\begin{lstlisting}
/-- The natural isomorphism between `nerveFunctor' and `nerveFunctor₂ ⋙ Truncated.cosk 2' whose components `nerve C ≅ (Truncated.cosk 2).obj (nerveFunctor₂.obj C)' show that nerves of categories are 2-coskeletal. -/
def !\cdocTwo[Nerve.]{cosk}{Iso}! : !\cdoc{nerveFunctor}! ≅ !\cdocTwo[Nerve.]{nerveFunctor}{}! ⋙ !\ldoc[SSet.]{Truncated.cosk}! 2 :=
  !\cdoc{NatIso.ofComponents}! (fun C ↦ (!\cdoc{nerve}! C).!\cdoc[SimplicialObject.]{isoCoskOfIsCoskeletal}! 2)
    (fun _ ↦ (!\ldoc[SSet.]{coskAdj}! 2).unit.naturality _)

/-- The canonical isomorphism `X ≅ (cosk n).obj X' defined when X is coskeletal and the n-coskeleton functor exists. -/
def !\cdoc[SimplicialObject.]{isoCoskOfIsCoskeletal}! [X.!\cdoc[SimplicialObject.]{IsCoskeletal}! n] : X ≅ (!\cdoc[SimplicialObject.]{cosk}! n).obj X :=
  asIso ((!\cdoc[SimplicialObject.]{coskAdj}! n).unit.app X)
\end{lstlisting}

A simplicial set $X$ is \emph{2-coskeletal} just when the unit component $\eta_X \colon X \to \cosk_2\,\tr_2\,X$ is an isomorphism. Thus, our first main subproblem, to prove that the natural transformation $\upsilon$ is a natural isomorphism, is achieved by the following proposition:

\begin{proposition}[{\ldoc[SSet.StrictSegal.]{isCoskeletal}}]\label{prop:2-coskeletal}
Nerves of categories are 2-coskeletal.
\end{proposition}
\begin{lstlisting}
!\href{https://github.com/leanprover-community/mathlib4/blob/bdf8f646f/Mathlib/AlgebraicTopology/SimplicialSet/Coskeletal.lean#L235-L236}{\color{keywordcolor}instance}! (C : Type u) [Category C] : (!\cdoc{nerve}! C).!\cdoc[SimplicialObject.]{IsCoskeletal}! 2 := ...
\end{lstlisting}

We unpack the 2-coskeletality condition and describe the proof of Proposition \ref{prop:2-coskeletal} in \S\ref{sec:strict-segal}.

\subsection{The 2-truncated nerve adjunction}\label{sec:ho2}

On account of the natural isomorphism $\upsilon \colon \nerve \cong \cosk_2\,\nerve_2$ provided by Proposition \ref{prop:2-coskeletal}, it now suffices to construct a left adjoint $\ho_2 \colon \Set^{\DDelta_{\leq 2}^\op} \to \Cat$ to the functor $\nerve_2$. To do so, we unpack the data contained in a (2-truncated) simplicial set.

On paper there is a well-known presentation of the morphisms in the \ldoc{SimplexCategory} $\DDelta$ by generators and relations \cite[\S I.1]{goerss-jardine} that is in the process of being formalized by Robin Carlier. At present, \libmathlib{} contains the generating \emph{face maps} $\delta^i \colon [n] \to [n + 1]$, corresponding to monomorphisms whose image omits $i \in \Fin(n + 2)$, as well as the generating \emph{degeneracy maps} $\sigma^i \colon [n + 1] \to [n]$, corresponding to epimorphisms for which $i \in \Fin(n+1)$ has two elements in its preimage; \libmathlib{} also knows the composition relations these maps satisfy, but does not contain a proof that $\DDelta$ is equivalent to the category \ldoc{SimplexCategoryGenRel} with this presentation.

A 2-truncated simplicial set $X$ is given by the data of three sets $X_0, X_1, X_2$ and maps between them indexed by the morphisms between the three objects $[0], [1], [2] : \DDelta_{\leq 2} \subset \DDelta$. The further 1-truncation of $X$  is given by the data below-left defined by the contravariant actions of the morphisms in $\DDelta_{\leq 1}$ displayed below-right:
\begin{equation}\label{eq:underlying-refl-quiv} \begin{tikzcd} X_0 \arrow[r, "\sigma_0" description] & X_1 \arrow[l, shift left=.6em, "\delta_1"] \arrow[l, shift right=.6em, "\delta_0"']  & & {[0]} \arrow[r, shift left=.6em, "\delta^0"] \arrow[r, shift right=.6em, "\delta^1"'] & {[1]} \arrow[l, "\sigma^0" description] \end{tikzcd}
\end{equation}
The composition relations    $\sigma^0 \circ \delta^0 = \id = \sigma^0 \circ \delta^1$ in $\DDelta$ induce dual relations $\delta_0 \circ \sigma_0 = \id = \delta_1 \circ \sigma_0$. As explained in Example \ref{ex:underlying-refl-quiv}, this is closely related to categorical data known as a \emph{reflexive quiver}.

A \emph{quiver} is defined in \libmathlib{} as a type equipped with a dependent type of arrows:
\begin{lstlisting}
class !\ldoc{Quiver}! (V : Type) where
  /-- The type of edges/arrows/morphisms between a given source and target. -/
  Hom : V → V → Type
\end{lstlisting}
Kim Morrison formalized the ``free $\dashv$ forgetful'' adjunction between the category $\Quiv$, of quivers and structure-preserving morphisms, called \emph{prefunctors}, between them, and $\Cat$.

\begin{proposition}[\cdoc{Quiv.adj}]\label{prop:quiv-adj}
  The free category on a quiver is left adjoint to the forgetful functor:
  $ \begin{tikzcd} \Cat \arrow[r, bend right=20, "U"' pos=.6] \arrow[r, phantom, "\bot"] & \Quiv\,. \arrow[l, bend right=20, "F"' pos=.4] \end{tikzcd} $
\end{proposition}

A similar result also holds for \emph{reflexive quivers}, which extend quivers with specified ``identity'' endo-arrows for each term in the base type:
\begin{lstlisting}
/-- A reflexive quiver extends a quiver with a specified arrow `id X : X ⟶ X' for each `X' in its type of objects. We denote these arrows by `id' since categories can be understood as an extension of reflexive quivers. -/
class !\cdoc{ReflQuiver}! (obj : Type) extends !\ldoc{Quiver}! obj where
  /-- The identity morphism on an object. -/
  id : ∀ X : obj, Hom X X
\end{lstlisting}

We introduced reflexive quivers and \emph{reflexive prefunctors} and defined the evident forgetful functors $U^r \colon \Cat \to \rQuiv$ from the category of categories to the category of reflexive quivers and $U^q \colon \rQuiv \to \Quiv$ from the category of reflexive quivers to the category of quivers. Definitionally, we have $U = U^q\,U^r \colon \Cat \to \Quiv$.


\begin{definition}[\cdoc{Cat.FreeRefl}] The \emph{free category} on a reflexive quiver $Q$ is the quotient of the free category $F\,U^q\,Q$ on the underlying quiver of $Q$ by the hom relation that identifies the specified reflexivity arrows in $Q$ with the identity arrows in the category $F\,U^q\,Q$.
\end{definition}

\begin{lstlisting}
/-- A reflexive quiver generates a free category, defined as as quotient of the free category on its underlying quiver (called the "path category") by the hom relation that uses the specified reflexivity arrows as the identity arrows. -/
def !\cdoc[Cat.]{FreeRefl}! (V) [!\cdoc{ReflQuiver}! V] :=
  Quotient (C := !\cdoc{Cat.free}!.obj (!\cdoc{Quiv.of}! V)) (!\cdoc[Cat.]{FreeReflRel}! (V := V))

instance (V) [!\cdoc{ReflQuiver}! V] : Category (!\cdoc[Cat.]{FreeRefl}! V) :=
  inferInstanceAs (Category (Quotient _))
\end{lstlisting}

This defines a functor $F^r \colon \rQuiv \to \Cat$ and a natural transformation $q \colon F\,U^q \Rightarrow F^r$ whose components $q_Q \colon F\,U^q\,Q \to F^r\,Q$ are universal among functors with domain $F\,U^q\,Q$ that respect the hom relation, sending the specified reflexivity arrows of $Q$ to identities.

We then formalized the following proof:
\begin{proposition}[\cdoc{ReflQuiv.adj}]\label{prop:free-refl-quiv-adj}
  The free category on a reflexive quiver is left adjoint to the forgetful functor:
  $ \begin{tikzcd} \Cat \arrow[r, bend right=20, "U^r"' pos=.6] \arrow[r, phantom, "\bot"] & \rQuiv\,. \arrow[l, bend right=20, "F^r"' pos=.4] \end{tikzcd} $
\end{proposition}
\begin{proof}
We construct the adjunction by defining the unit and counit and verifying the triangle identities. The component of the counit $\epsilon^r_C$ at a category $C$ is defined by appealing to the universal property of the quotient functor $q_{U^rC}$ in the following diagram:
\[ \begin{tikzcd} F\,U^q\, U^r\, C \arrow[r, "q_{U^rC}"] \arrow[d, equals] & F^r\,U^r\, C \arrow[d, dashed, "\epsilon^r_C"] \\ F\,U\,C \arrow[r, "\epsilon_C"'] & C \end{tikzcd}\]
Modulo the definitional equality $U^q\,U^r=U$, the counit component $\epsilon_C$ of the adjunction $F \dashv U$ sends the specified reflexivity arrows of $U^r\,C$, the identities of $C$, to identities. Thus this functor factors through the quotient functor $q_{U^rC}$ to define the counit component $\epsilon^r_C$. Naturality of $\epsilon$ follows from the uniqueness of the universal property of the quotient functor.

The construction of the unit components at a reflexive quiver $Q$ is more subtle. In the category $\Quiv$ of quivers and prefunctors we have the following composite: \[ \begin{tikzcd} U^q\,Q \arrow[r, "\eta_{U^q\,Q}"] & U\,F\,U^q\,Q \arrow[r, "U\,q_Q"] & U\,F^r\,Q  = U^q\,U^rF^r\,Q, \end{tikzcd}\] defining a prefunctor between the underlying quivers of the reflexive quivers $Q$ and $U^rF^rQ$. In fact, this map preserves the specified reflexivity arrows and hence lifts to a reflexive prefunctor $\eta^r_Q \colon Q \to U^r F^r Q$ with defining equality:
\[ \begin{tikzcd} U^q\, Q \arrow[d, "\eta_{U^q\,Q}"'] \arrow[r, "U^q\,\eta^r_Q", dashed] & U^q\,U^rF^r\,Q \arrow[d, equals] \\ U\,F\,U^q\,Q\arrow[r, "U\,q_Q"'] & U\,F^r\,Q\, . \end{tikzcd}\]
Naturality in this case actually holds by reflexivity.

The proofs of the triangle equalities require diagram chases. To verify commutativity of
\[ \begin{tikzcd} U^r\, C \arrow[r, "\eta^r_{U^r\,C}"] \arrow[rr, bend right=15, equals] & U^r\,F^r\,U^r\,C \arrow[r, "U^r\,\epsilon^r_C"] & U^r\,C, \end{tikzcd}\] we use the fact that the functor $U^q \colon \rQuiv \to \Quiv$ is faithful and the fact that $U^q$ of this composite is equal to the corresponding triangle identity composite for the adjunction $F \dashv U$. To verify commutativity of
\[ \begin{tikzcd} F^r\, Q \arrow[r, "F^r \eta^r_Q"] \arrow[rr, bend right=15, equals] & F^r\,U^r\,F^r\,Q \arrow[r, "\epsilon^r_{F^rQ}"] & F^r\, Q, \end{tikzcd}\] we use the universal property of the quotient functor to instead verify commutativity of
\[ \begin{tikzcd} F\,U^q\, Q \arrow[r, "q_Q"] \arrow[rrr, bend right=10, "q_Q"'] & F^r\, Q \arrow[r, "F^r \eta^r_Q"] & F^r\,U^r\,F^r\,Q \arrow[r, "\epsilon^r_{F^rQ}"] & F^r\, Q, \end{tikzcd}\] which follows from naturality of $q$, the defining equalities for $\epsilon^r$ and $\eta^r$, naturality of $\epsilon$, and the corresponding triangle equality for $F \dashv U$. These steps occupy 7 lines of a 23 line proof due to the challenges described in \S\ref{sec:expected-challenges}.
\end{proof}


\begin{example}[{\ldocTwo[SSet.]{instReflQuiverOneTruncation}{}}]\label{ex:underlying-refl-quiv}
  Returning to \eqref{eq:underlying-refl-quiv}, a 2-truncated simplicial set $X$ has an underlying reflexive quiver $U^2_1X$ whose type of objects is $X_0$ and whose type of arrows from $x : X_0$ to $y : X_0$ is the type of 1-simplices $f : X_1$ so that $\delta_1 f = x$ and $\delta_0 f = y$.\footnote{To explain the orientation, note that $\delta^1 \colon [0] \to [1]$ picks out the source object of $\Fin(2)$ while $\delta^0 \colon [0] \to [1]$ picks out the target object.}
\end{example}

We can now define the 2-truncated homotopy category functor.

\begin{definition}[{\ldoc[SSet.]{Truncated.HomotopyCategory}}]\label{defn:homotopy-category}
The homotopy category $\ho_2\,X$ is the quotient of the free category $F^r\, U^2_1\,X$ on the underlying reflexive quiver of a 2-truncated simplicial set $X$ by a hom relation generated by the 2-simplices: for every 2-simplex $\sigma : X_2$, we identify the face $\delta_1\, \sigma$ with the composite $\delta_0\,\sigma \cdot \delta_2\, \sigma$ in the category $F^r\,U^2_1\,X$.
\end{definition}
\begin{lstlisting}
/-- The 2-simplices in a 2-truncated simplicial set `V' generate a hom relation on the free category on the underlying reflexive quiver of `V'. -/
inductive !\ldocTwo[SSet.Truncated.]{HoRel}{}! {V : !\ldoc{SSet.Truncated}! 2} :
    (X Y : !\cdoc{Cat.FreeRefl}! (!\ldocTwo[SSet.]{OneTruncation}{}! V)) → (f g : X ⟶ Y) → Prop
  | mk (φ : V _⦋2⦌₂) : !\ldocTwo[SSet.Truncated.]{HoRel}{}! _ _
                        (Quot.mk _ (!\ldoc{Quiver.Hom.toPath}! (!\ldocTwo[SSet.Truncated.]{ev02}{}! φ)))
                        (Quot.mk _ ((!\ldoc{Quiver.Hom.toPath}! (!\ldocTwo[SSet.Truncated.]{ev01}{}! φ)).comp
                          (!\ldoc{Quiver.Hom.toPath}! (!\ldocTwo[SSet.Truncated.]{ev12}{}! φ))))

/-- The type underlying the homotopy category of a 2-truncated simplicial set V. -/
def !\ldoc{SSet.Truncated.HomotopyCategory}! (V : !\ldoc{SSet.Truncated}! 2) : Type u :=
  Quotient (!\ldocTwo[SSet.Truncated.]{HoRel}{}! (V := V))

instance (V : !\ldoc{SSet.Truncated}! 2) : Category.{u} (V.!\ldoc[SSet.Truncated.]{HomotopyCategory}!) :=
  inferInstanceAs (Category (!\ldoc{CategoryTheory.Quotient}! ..))
\end{lstlisting}

This defines a functor $\ho_2 \colon \Set^{\DDelta_{\leq 2}^\op} \to \Cat$ and a natural transformation $q \colon F^r\, U^2_1 \Rightarrow \ho_2$ whose components $q_X \colon F^r\, U^2_1\, X \to \ho_2\,X$ are universal among functors with domain $F^r\, U^2_1\, X$ that respect the hom relation, sending the image of the composite $\delta_0\, \sigma \cdot \delta_2\, \sigma$ of the arrows along the ``spine'' of a 2-simplex $\sigma : X_2$ to the image of the ``diagonal'' arrow $\delta_1\, \sigma$.

The proof that $\ho_2$ is left adjoint to $\nerve_2$ proceeds similarly to the proof of Proposition~\ref{prop:free-refl-quiv-adj}, with one step made considerably more difficult. As before, the counit of $\ho_2 \dashv \nerve_2$ is defined using the universal property of the quotient functor $q$ while the unit is defined by lifting a map from the category of reflexive quivers to the category of 2-truncated simplicial sets. This lifting is enabled by the following proposition:

\begin{proposition}[\cdocTwo{toNerve}{.mk}, \cdocTwo{toNerve}{.ext}]\label{prop:to-nerve2-mk}
Let $X : \Set^{\DDelta_{\leq 2}^\op}$ and $C : \Cat$.

\begin{romanenumerate}
\item Consider a reflexive prefunctor $F : U^2_1\, X \to U^2_1\, \nerve_2\, C \cong U^r\, C$ from the underlying reflexive quiver of $X$ to the underlying reflexive quiver of $C$. If for each 2-simplex $\sigma : X_2$, $F(\delta_1\, \sigma)$ equals the composite $F(\delta_0\, \sigma) \cdot F(\delta_2\, \sigma)$ in $C$, then $F$ lifts to a map of 2-truncated simplicial sets.
\item Any parallel pair of maps  of 2-truncated simplicial sets $F,G \colon X \to \nerve_2\, C$ are equal if their underlying reflexive prefunctors agree.
\end{romanenumerate}
\end{proposition}

We discuss the proof of these results in \S\ref{sec:naturality}.

\begin{proposition}[\cdocTwo{nerve}{Adj}]\label{prop:2-truncated-nerve-adj}
The 2-truncated homotopy category functor is left adjoint to the 2-truncated nerve:
$ \begin{tikzcd} \Cat\ \arrow[r, bend right=20, "\nerve_2"' pos=.6] \arrow[r, phantom, "\bot"] & \Set^{\DDelta_{\leq 2}^\op}. \arrow[l, bend right=20, "\ho_2"' pos=.4] \end{tikzcd}$
\end{proposition}
\begin{proof}
We construct the adjunction by defining the unit and counit and verifying the triangle identities. The component of the counit $\epsilon_
C$ at a category $C$ is defined by appealing to the universal property of the quotient functor $q_{\nerve_2\, C}$ in the following diagram:
\[ \begin{tikzcd}[column sep=large] F^r\, U^2_1\, \nerve_2\, C \arrow[r, "q_{\nerve_2\, C}"] \arrow[d, "F^r\phi_C"', "\cong"] & \ho_2\, \nerve_2\, C \arrow[d, dashed, "\epsilon_C"] \\ F^r\, U^r\, C \arrow[r, "\epsilon^r_C"'] & C, \end{tikzcd}\]
Modulo a natural isomorphism $\phi \colon U^2_1\,\nerve_2 \cong U^r$ between the two constructions of the reflexive quiver underlying a category, the counit component $\epsilon^r_C$ of the adjunction $F^r \dashv U^r$ respects the defining hom relation. Naturality is again established using the uniqueness of the universal property of the quotient functor.

The unit components are constructed by applying Proposition \ref{prop:to-nerve2-mk}(i) to the composition of reflexive prefunctors defined below, which thus satisfies the defining equation:
\[ \begin{tikzcd} U^2_1\, X \arrow[r, "U^2_1 \eta_X", dashed ] \arrow[d, "\eta^r_{U^2_1X}"'] & U^2_1\, \nerve_2\, \ho_2\, X \arrow[d, "\phi_{\ho_2 X}", "\cong"'] \\ U^r\, F^r\, U^2_1\, X \arrow[r, "U^rq_X"'] & U^r\, \ho_2\, X. \end{tikzcd} \] By Proposition \ref{prop:to-nerve2-mk}(ii), naturality can be checked at the level of reflexive prefunctors, at which point it follows from naturality of $\eta^r$, $q$, and $\phi$.

The proofs of the triangle equalities require diagram chases. To verify commutativity of
\[ \begin{tikzcd}[sep=large] \nerve_2\, C \arrow[r, "\eta_{\nerve_2\,C}"] \arrow[rr, bend right=15, equals] & \nerve_2\,\ho_2\,\nerve_2\, C \arrow[r, "\nerve_2\,\epsilon_C"] & \nerve_2\,C, \end{tikzcd}\]
we use Proposition \ref{prop:to-nerve2-mk}(ii) to instead verify commutativity of the underlying reflexive prefunctors. This diagram chase follows from the defining equalities for $\eta$ and $\epsilon$, naturality of $\phi$ and $\eta^r$, and   the corresponding triangle equality for $F^r \dashv U^r$, modulo cancelling $\phi$ with its inverse.  These steps occupy 8 lines of a 25 line proof due to the challenges described in \S\ref{sec:expected-challenges}.

To verify commutativity of
\[ \begin{tikzcd} \ho_2\, X \arrow[r, "\ho_2 \eta_X"] \arrow[rr, bend right=15, equals] & \ho_2\,\nerve_2\,\ho_2\,X \arrow[r, "\epsilon_{\ho_2\,X}"] & \ho_2\, X, \end{tikzcd}\] we use the universal property of the quotient functor to instead verify commutativity of
\[ \begin{tikzcd} F^r\,U^2_1\, X \arrow[r, "q_X"] \arrow[rrr, bend right=10, "q_X"'] &  \ho_2\, X \arrow[r, "\ho_2 \eta_X"]  & \ho_2\,\nerve_2\,\ho_2\,X \arrow[r, "\epsilon_{\ho_2\,X}"] & \ho_2\, X,  \end{tikzcd}\] which follows from naturality of $q$, the defining equalities for $\eta$ and $\epsilon$, naturality of $\epsilon^r$, and  the corresponding triangle equality for $F^r \dashv U^r$, modulo cancelling $\phi$ with its inverse. These steps occupy 8 lines of a 24 line proof due to the challenges described in \S\ref{sec:expected-challenges}.
\end{proof}

As discussed, Propositions \ref{prop:2-truncated-nerve-adj} and \ref{prop:2-coskeletal} imply Theorem \ref{thm:nerveAdj}, proving our first main theorem. To show that the adjunction $\ho \dashv \nerve$ is reflective and conclude that $\Cat$ has colimits, we must show that the nerve functor is fully faithful. As the nerve is isomorphic to the composite of the functors $\cosk_2$ and $\nerve_2$ and $\cosk_2$ is fully faithful (since $\DDelta_{\leq 2} \hookrightarrow \DDelta$ is the inclusion of a full subcategory), it suffices to show that the 2-truncated nerve is fully faithful.

\begin{lemma}[\cdocTwo{nerveFunctor}{.faithful}]
  The 2-truncated nerve functor is faithful.
\end{lemma}
\begin{proof}
  We have a natural isomorphism $\phi \colon U_1^2\, \nerve_2 \cong U^r$ and the functor $U^r$ is faithful. Thus, as the first functor of a faithful composite, $\nerve_2$ is faithful.
\end{proof}
\begin{lstlisting}
instance !\cdocTwo{nerveFunctor}{.faithful}! : !\cdocTwo[Nerve.]{nerveFunctor}{}!.!\cdoc[Functor.]{Faithful}! :=
  !\cdoc{Functor.Faithful.of\_comp\_iso}!
    (G := !\ldocTwo[SSet.]{oneTruncation}{}!) (H := !\cdoc{ReflQuiv.forget}!) !\ldocTwoTwo[SSet.]{OneTruncation}{.ofNerve}{.natIso}!
\end{lstlisting}

With considerably more effort, discussed in \S\ref{sec:full}, we can also show:

\begin{lemma}[\cdocTwo{nerveFunctor}{.full}]\label{lem:nerve2-full}
The 2-truncated nerve functor is full.
\end{lemma}

These lemmas combine to prove Theorem \ref{thm:nerveFF}, our second main theorem, and these results together give Corollary \ref{cor:cat-colimits} that $\Cat$ has colimits.

\section{Challenges}\label{sec:challenges}

Although the mathematics of Theorems \ref{thm:nerveAdj} and \ref{thm:nerveFF} is well established in the literature, and we came to the project suitably armed with category theoretic and Lean formalization expertise, the project ended up being quite challenging, for both expected and unexpected reasons. Below, we will discuss some of the aspects of the formalization that posed the most difficulties, and what we did to resolve or work around the issues.

\subsection{Bundling categories}\label{sec:expected-challenges}



In \libmathlib, many notions come in two flavors of presentation, referred to as the ``bundled'' and ``(partially) unbundled'' presentations.\footnote{There is also a ``fully unbundled'' presentation, but \libmathlib{} does not use it and we will not discuss it.} In the unbundled presentation, a category is given as a type variable $C$ together with a typeclass instance argument \texttt{\cdoc{Category}\;$C$} which provides the rest of the data for a category whose type of objects is $C$. This allows for a form of formal synecdoche, where one need only explicitly mention the type of objects and the rest of the structure comes along for the ride. In bundled style a category is an element $C:\cdoc{Cat}$, with suitable projections $\ob\;C:\mathsf{Type}$ and $\Hom:\ob\;C\to\ob\;C\to\mathsf{Type}$.

Generally, \libmathlib{} prefers the unbundled form, but the bundled form is necessary if we need to consider the category of categories, that is, \texttt{\cdoc{Category}\;\cdoc{Cat}}, and our main theorem is about \cdoc{Cat} so we cannot avoid it completely. A similar situation occurs for functors, which appear unbundled as $C\rightharpoons D$ and also bundled as $\Hom_{\cdoc{Cat}}(C,D)$.\footnote{Note the former notion of functor is strictly speaking more general than the latter, in which the universe levels of the domain category must match those of the codomain.} There is functor composition $F\ggg G$ and morphism composition $F\gg_{\cdoc{Cat}} G$, and so on. Not only are many notions duplicated, but by applying lemmas coming from each context one can end up with a mixture of both notions in goals, and simplification would frequently get stuck in this situation. We eventually mitigated this issue by unbundling as early as possible, stating lemmas only in the unbundled form, and packaging them up only for the higher order statements about operations on \cdoc{Cat}.

\subsection{Equality of functors}\label{sec:functor-ext}

Our main theorem fundamentally deals with the 1-categorical (rather than 2-categorical) structure of $\Cat$. The very statement that $\Cat$ has small colimits involves equality of functors, as does the definition of the nerve --- a simplicial \emph{set} rather than a simplicially-indexed category-valued pseudofunctor.\footnote{The nerve of $C$ could be regarded as a pseudofunctor $\nerve\;C \colon \DDelta^\op \to \Cat$ valued in $\Cat$, sending the object $[n] : \DDelta$ to the category of functors $\Fin(n+1) \to C$, but this extra coherence data is unnecessary. The type-valued mapping  $\nerve\;C \colon \DDelta^\op \to \Set$ is strictly functorial, and indeed (due to the care taken in formalizing the definition of \cdoc{ComposableArrows}) the \lstinline|map_id| and \lstinline|map_comp| fields of \cdoc{nerve} are filled by \lstinline|rfl|.} There are advantages to working strictly (1-categorically) rather than weakly (2-categorically or bicategorically) when possible. In future work, described in \S\ref{sec:future}, we plan to use our homotopy category functor to convert from simplicial to categorical enrichments. As we have formalized this as a strict functor between 1-categories we can tap into existing \libmathlib{} developments of change of base in enriched category theory.

Lean's \libmathlib{} provides an extensionality theorem for functors, contributed by Reid Barton, but note the warning:
\begin{lstlisting}
/-- Proving equality between functors. This isn't an extensionality lemma, because usually you don't really want to do this. -/
theorem !\cdoc[Functor.]{ext}! {F G : C ⥤ D} (h_obj : ∀ X, F.obj X = G.obj X) (h_map : ∀ X Y f,
    F.map f = !\cdoc{eqToHom}! (h_obj X) ≫ G.map f ≫ !\cdoc{eqToHom}! (h_obj Y).symm) : F = G
\end{lstlisting}
So we knew in advance there would be trouble coming and needed to structure our proofs in specific ways to take advantage of \cdoc{eqToHom}.

At the core of this warning is the following issue: When one has an equality of objects, or an equality of functors (which involves an equality of objects --- the \lstinline|h_obj| assumption above), this induces an equality of types, namely $\Hom(A,B)=\Hom(A,B')$ given $B=B'$. In dependent type theory, if one has $f:\Hom(A,B)$ and an equality $h:B=B'$, the plain expression $f:\Hom(A,B')$ does not typecheck (unless $B$ and $B'$ are \emph{definitionally equal}). However, it is possible to use $h$ to \emph{cast} $f$ to the correct type, producing a term $\cast\,h\,f:\Hom(A,B')$. This is also known as ``rewriting'' and it is a common technique in proofs (using the \lstinline|rw| tactic), but it is hazardous when applied to ``data'' such as $f$, because $\cast\,h\,f$ is not the same term as $f$ and so mismatches can arise between expressions such as $\cast\,h\,(f\circ g)$ and $(\cast\,h\,f)\circ g$. It is possible to prove these are equal, but it generally requires a rather delicate ``induction on equality'' to prove, with a manually written induction motive. Complications with transport of this nature are often referred to as ``DTT hell.''

The \libmathlib{} statement of \cdoc{Functor.ext} is employing a technique to avoid this issue, which is to make use of the category structure to help alleviate the pains associated with using the $\cast$ function directly, by first converting $h:B=B'$ into $\cdoc{eqToHom}\,h:\Hom(B,B')$, so that $(\cdoc{eqToHom}\,h) \circ f:\Hom(A,B')$. This presents the casting of $f$ as a composition, which means that many more lemmas from category theory apply (such as associativity of composition), and \cdoc{eqToHom} itself has many lemmas for how it factors over various operations.

For the most part, this technique did its job. Although we obtained many goals that involved these \cdoc{eqToHom} morphisms, after formalizing appropriate lemmas, the simplifier was able to push them out of the way. For instance, we added a \ldoc{Quiver.homOfEq} to perform an analogous casting for (reflexive) quivers (which lack composition) and proved various lemmas by induction on equality. This is the kind of complication that essentially never appears in informal category theory, due to its more extensional handling of morphism typing.

\subsection{Proving functoriality}\label{sec:full}

Another example of DTT hell arises in the proof of \cdocTwo{nerveFunctor}{.full}, that the 2-truncated nerve functor $\nerve_2\colon \Cat \to \Set^{\DDelta_{\leq 2}^\op}$ is full. Given a pair of categories $C$ and $D$ and a map $F \colon \nerve_2\, C \to \nerve_2\, D$ between their 2-truncated nerves, our task is to lift this to a functor between $C$ and $D$. The underlying reflexive prefunctor is defined by the 1-truncation of $F$, but it remains to show that this reflexive prefunctor is \emph{functorial}, preserving composition of morphisms in $C$ in addition to identities.

The idea of the proof is to note that the composable morphisms $k \circ h$ define a 2-simplex $\sigma_{k,h} : (\nerve_2\,C)_2$. The 2-simplex $F(\sigma_{k,h}) : (\nerve_2\,D)_2$ defines a functor $\Fin(3) \to D$, which gives rise to a commutative triangle involving three arrows in $D$. Using the naturality of $F$ with respect to the three morphisms $\delta^i \colon [1] \to [2]$, this may be converted into the desired equality $F(k \circ h) = Fk \circ Fh$. But these naturality equalities live in the type $(\nerve_2\,D)_1$ of functors $\Fin(2) \to D$, which give rise to \emph{heterogeneous} equalities between the corresponding arrows in $D$. We were forced to first prove and then string together a sequence of heterogeneous equalities before finally concluding the desired equality between arrows belonging to the same hom-type. It then remained to show that the 2-truncated nerve carries this lifted functor to the original map $F$, which is proven using Proposition \ref{prop:to-nerve2-mk}(ii).

\subsection{Proving naturality}\label{sec:naturality}

In Proposition \ref{prop:to-nerve2-mk}(i), our aim is to define a map of 2-truncated simplicial sets $F : X \to \nerve_2\, C$ from a map between the underlying reflexive quivers. Such a map is a natural transformation with three components $F_0$, $F_1$, and $F_2$ corresponding to the three objects $[0], [1], [2] : \DDelta_{\leq 2}$, which can be constructed straightforwardly from the hypotheses. But complications arise in the proof of naturality, which amounts to establishing commutative squares of the form
\begin{equation}\label{eq:to-nerve-naturality}
  \begin{tikzcd} X_n \arrow[r, "F_n"] \arrow[d, "{-\cdot\alpha}"'] & (\nerve_2\,C)_n \arrow[d, "{-\cdot\alpha}"] \\ X_m \arrow[r, "F_m"'] & (\nerve_2\,C)_m  \end{tikzcd}
\end{equation}
for each morphism $\alpha \colon [m] \to [n]$ in $\DDelta_{\leq 2}$ (there are 31 such morphisms). One annoyance, discussed at the start of \S\ref{sec:ho2}, is the fact that \libmathlib{} does not yet know that $\DDelta_{\leq 2}$ is generated by the three degeneracy maps $\sigma^i$ and the five face maps $\delta^i$. Our initial formalization instead explicitly treated the nine cases presented by the choice of domain and codomain objects, with the cases involving a morphism with domain $[2]$ reducing to the cases involving domain $[0]$ or $[1]$ by the 2-coskeletality of the nerve. \fixme{Really what I want to cite here is the fact that \cdocTwo{nerve}{seagull} is a monomorphism but that instance has a horrible name. Can you help?} Once we fixed domain and codomain objects there were often further case splits involving the specific nature of the morphism $\alpha \colon [m] \to [n]$.

When we submitted this as a pull request to \libmathlib{}, the reviewer Jo\"{e}l Riou suggested we think of this result as the task of showing that the \cdoc{MorphismProperty} of satisfying the naturality condition \eqref{eq:to-nerve-naturality} holds for every morphism $\alpha$ in $\DDelta_{\leq 2}$. An easier task than showing that the simplex category can be presented by generators and relations is to show that the \cdoc[MorphismProperty.]{naturalityProperty} of its arrows follows from the cases corresponding to faces and degeneracy maps, as Riou formalized for us in \ldoc[SimplexCategory.]{Truncated.morphismProperty\_eq\_top}. This reduces the problem to the eight cases corresponding to generating morphisms $\alpha \colon [m] \to [n]$ in $\DDelta_{\leq 2}$, each of which still involves establishing an equality between functors $\Fin(m+1) \to C$.

The extensionality result of Proposition \ref{prop:to-nerve2-mk}(ii) is simpler but still less innocuous than it appears. Two natural transformations $F$ and $G$ are equal just when their three components are equal, and as these are functions between types it suffices to show they define the same mapping on terms. However, this involves proving equalities in the types $(\nerve_2\,C)_0$, $(\nerve_2\,C)_1$, and $(\nerve_2\,C)_2$, which are types of functors, and thus these equalities involved a fair amount of pain. In the final case $F_2 = G_2$, naturality again provided us with equalities in the type $(\nerve_2\,C)_1$ of functors $\Fin(2) \to C$ which gave rise to heterogeneous equalities between the mappings on arrows of the functors $\Fin(3) \to C$ we were tasked with identifying.

\subsection{Proving nerves are 2-coskeletal}\label{sec:strict-segal}


By a folklore result involving Reedy category theory \cite{Reedy}, a simplicial set $X$ is $2$-\emph{coskeletal}, meaning that the adjunction unit component $\eta_X\colon X \to \cosk_2\, \tr_2\, X$ is an isomorphism, just when every $k$-simplex boundary in $X$ can be filled to a simplex, for all $k > 2$. This property would be easy to check in the case where $X$ is the nerve of a category, but unfortunately this characterization of 2-coskeletal simplicial sets has not been formalized. Thus, we directly unpack the invertible unit component definition into a condition that we can actually check.



Recall the functor $\cosk_2$ is defined by right Kan extension along the inclusion $\DDelta_{\leq 2}^\op \hookrightarrow \DDelta^\op$. Thus, the condition of a simplicial set $X$ being 2-coskeletal asserts that the identity natural transformation defines a right Kan extension:
\begin{equation}\label{eq:cosk-kan-ext} \begin{tikzcd} & \DDelta^\op \arrow[dr, "X"] \arrow[d, phantom, "\Downarrow\id"]\\ \DDelta_{\leq 2}^\op \arrow[ur, hook] \arrow[r, hook] & \DDelta^\op \arrow[r, "X"'] & \Set. \end{tikzcd}
\end{equation}

\begin{lstlisting}
/-- A simplicial object `X' is `n'-coskeletal when it is the right Kan extension of its restriction along `(Truncated.inclusion n).op' via the identity natural transformation. -/
class !\cdoc[SimplicialObject.]{IsCoskeletal}! (X : !\cdoc{SimplicialObject}! C) (n : ℕ) : Prop where
  isRightKanExt : !\cdoc[Functor.]{IsRightKanExtension}! X (𝟙 ((!\ldoc[SimplexCategory.]{Truncated.inclusion}! n).op ⋙ X))

theorem !\cdoc[SimplicialObject.]{isCoskeletal\_iff\_isIso}! :
    X.!\cdoc[SimplicialObject.]{IsCoskeletal}! n ↔ !\cdoc{IsIso}! ((!\cdoc[SimplicialObject.]{coskAdj}! n).unit.app X) := by
  rw [!\cdoc[SimplicialObject.]{isCoskeletal\_iff}!]
  exact !\cdoc[Functor.]{isRightKanExtension\_iff\_isIso}! ((!\cdoc[SimplicialObject.]{coskAdj}! n).unit.app X)
    ((!\cdoc[SimplicialObject.]{coskAdj}! n).counit.app _) (𝟙 _) ((!\cdoc[SimplicialObject.]{coskAdj}! n).left_triangle_components X)
\end{lstlisting}
As the category $\Set$ is complete, any such right Kan extension \eqref{eq:cosk-kan-ext} is \emph{pointwise}, meaning that $X$ is 2-coskeletal if and only if for each $n : \NN$, a canonical cone --- with summit $X_n$ over a diagram with objects $X_j$ indexed by arrows $[j] \to [n]$ in $\DDelta$ with $j \leq 2$ --- is a limit cone, the definition of which we now recall:
\begin{lstlisting}
/-- A cone `t' on `F' is a limit cone if each cone on `F' admits a unique cone morphism to `t'. -/
structure !\cdoc[Limits.]{IsLimit}! (t : !\cdoc[Limits.]{Cone}! F) where
  /-- There is a morphism from any cone point to `t.pt' -/
  lift : ∀ s : !\cdoc[Limits.]{Cone}! F, s.pt ⟶ t.pt
  /-- The map makes the triangle with the two natural transformations commute -/
  fac : ∀ (s : !\cdoc[Limits.]{Cone}! F) (j : J), lift s ≫ t.π.app j = s.π.app j
  /-- It is the unique such map to do this -/
  uniq : ∀ (s : !\cdoc[Limits.]{Cone}! F) (m : s.pt ⟶ t.pt)
    (_ : ∀ j : J, m ≫ t.π.app j = s.π.app j), m = lift s
\end{lstlisting}
Thus, for the simplicial set $\nerve\ C$, we must:
\begin{romanenumerate}
\item Construct \texttt{lift}, a function valued in functors $\Fin(n + 1) \to C$.
\item Prove \texttt{fac}, an equality between functions valued in functors $\Fin(j + 1) \to C$ for $j = 0, 1, 2$.
\item Prove \texttt{uniq}, an equality between functions valued in functors $\Fin(n + 1) \to C$.
\end{romanenumerate}

We submitted a proof along these lines in our initial \libmathlib{} pull request, involving repeated painful appeals to \cdoc{Functor.ext}. During the PR review process, Jo\"{e}l Riou suggested a useful abstraction boundary, using the following characterization of simplicial sets defined by nerves of categories.

\begin{proposition}[\cdoc{Nerve.strictSegal}]\label{prop:strict-segal}
  For a simplicial set $X$ the following are equivalent:
  \begin{romanenumerate}
    \item $X$ is isomorphic to the nerve of a category.\footnote{In fact, when this condition holds, $X$ is isomorphic to the nerve of the homotopy category $\ho(X)$.}
    \item $X$ satisfies the \ldoc[SSet.]{StrictSegal} condition: for all $n : \NN$, the canonical map that carries an $n$-simplex of $X$ to the path of edges defined by its \ldoc[SSet.]{spine} is an equivalence.
  \end{romanenumerate}
\end{proposition}

As Proposition \ref{prop:strict-segal} explains, the only examples of strict Segal simplicial sets are nerves of categories. But it was cleaner to formalize the proof that a strict Segal simplicial set $X$ is 2-coskeletal because the functions involved in \lstinline|lift|, \lstinline|fac|, and \lstinline|uniq| are now valued in an abstract type $X_n$ which is equivalent to the type of paths of edges of length $n$ in $X$. Thus, we formalized:
\begin{lstlisting}
/-- A strict Segal simplicial set is 2-coskeletal. -/
def !\ldoc[SSet.StrictSegal.]{isPointwiseRightKanExtensionAt}! (n : ℕ) :
      (!\ldoc[SSet.Truncated.]{rightExtensionInclusion}! X 2).!\cdoc[Functor.RightExtension.]{IsPointwiseRightKanExtensionAt}! ⟨⦋n⦌⟩ where
  lift s x := ...
  fac s j := ...
  uniq s m hm := ...
\end{lstlisting}
together with a proof \cdoc{Nerve.strictSegal} of the implication (i)$\Rightarrow$(ii) in Proposition \ref{prop:strict-segal}.

\section{Related work}\label{sec:related-work}

There are myriad libraries of formalized mathematics that contain a substantial amount of category theory \cite{ BauerGrossLumsdaineShulmanSozeauSpitters2017, hott-in:agda, HuCarette2021, 1Lab, mathlib2020, CZH_Foundations-AFP,CZH_Elementary_Categories-AFP,CZH_Universal_Constructions-AFP, agda-unimath, VezzosiMoertbergAbel2021, UniMath}. Several papers on the subject \cite{GrossChlipalaSpivak, OKeefe2005} recall Harrison's assessment, based on the early efforts of \cite{HS-Constructive}, that ``category theory [is] notoriously hard to formalize in any kind of system'' \cite{harrison1996formalized}. Others describe additional complexities involved in developing a category theory library that is compatible with intensional identity types or homotopy type theory \cite{AKS-Univalent, HuCarette2021}.

A common feature of category theory libraries is a broad array of sophisticated categorical definitions, constructions, and theorems. To give a simplified illustration related to the theory of colimits, one will commonly find the general definition of a colimit, specific examples of colimits such as coproducts and coequalizers, and a proof that a colimit of any diagram shape can be built out of coequalizers and suitably-sized coproducts. What is less common is to find formalizations of applications of this general theory to specific categories, e.g., a proof that a specific category has coproducts and coequalizers and thus has all suitably-sized colimits --- with a common exception being the category of types and functions.\footnote{The \liboneLab{} also contains a formalization of limits and colimits in type-valued presheaf categories: \href{https://1lab.dev/Cat.Instances.Presheaf.Colimits.html}{\tt 1lab.dev/Cat.Instances.Presheaf.Colimits.html}}

More extensive examples of applications of category theory to specific categories can be found in the \libunimath{} library in Rocq \cite{UniMath}, with the categorical content originating from \cite{AKS-Univalent}, as well as Lean's \libmathlib{} \cite{mathlib2020}, whose category theory formalizations grew out of an experimental library developed by Kim Morrison \cite{Morrison}. Both libraries contain the construction of limits in $\Cat$. More explicitly, \libunimath{} contains a construction of equalizers between so-called ``strict categories'' (those whose objects and morphisms form \emph{sets}, rather than higher types).\footnote{See: \href{https://github.com/UniMath/UniMath/blob/master/UniMath/CategoryTheory/Limits/Examples/CategoryOfSetcategoriesLimits.v}{\tt UniMath/CategoryTheory/Limits/Examples/CategoryOfSetcategoriesLimits.v}} By contrast, \libmathlib{} directly formalizes the construction of the limit of a diagram of any shape as the category whose objects and hom-sets are defined by limits of the induced diagrams.\footnote{See: \href{https://leanprover-community.github.io/mathlib4_docs/Mathlib/CategoryTheory/Category/Cat/Limit.html}{\tt Mathlib/CategoryTheory/Category/Cat/Limit.html}} We do not know of any other formalizations of the nerve functor, the nerve adjunction, or colimits in $\Cat$.

Alternate pen-and-paper approaches to the construction of colimits in $\Cat$ can be given. Since coproducts in $\Cat$ are elementary to define --- created by the forgetful functor to reflexive quivers --- this reduces the problem to the construction of coequalizers. In \cite[5.1.7]{Borceux:1994hbI}, coequalizers are constructed in reference to the adjunction of Proposition \ref{prop:quiv-adj} lifted to the category of quivers with commutativity conditions.  In \cite{BBP:Generalized}, which Robert Maxton brought to our attention after our formalization was complete, coequalizers in $\Cat$ are constructed as quotient functors defined with respect to what the authors call a \emph{generalized congruence}, generalizing the composition-stable hom-wise equivalence relations called \emph{congruences} that we use here. This approach provides an explicit construction of coequalizers in $\Cat$ that would be worth formalizing, though the construction of general colimits as coequalizers of maps between coproducts would still suffer from the unnecessary complexity described in the introduction. A final approach that we considered would be to observe that $\Cat$ is the category of models of a finite limit theory, and thus is an essentially algebraic theory, and all such categories are locally presentable, and in particular cocomplete \cite[\S 3.D]{AdamekRosicky}. The construction of colimits given in the proof of this result is rather unwieldy \cite[3.36]{AdamekRosicky}. As noted above, our approach follows \cite[4.5.16]{Riehl:2016cc}.

\section{Future work}\label{sec:future}

There are several avenues for future work pursuing the various ways in which the formalizations undertaken here could be extended, generalized, or applied. We plan to make use of the homotopy category functor to formalize a proof of the converse implication (ii)$\Rightarrow$(i) in Proposition \ref{prop:strict-segal} by showing that a strict Segal simplicial set is canonically isomorphic, via the unit of our nerve adjunction, to the nerve of its homotopy category. At our suggestion, Nick Ward has formalized the $n$-truncated analog of strict Segal simplicial sets. 
Using this, we also plan to refactor our construction of the unit of the nerve adjunction through the abstraction boundary of strict Segal 2-truncated simplicial sets, which will allow us to consolidate, though not totally remove, the uses of \cdoc{Functor.ext}.

The proofs of Propositions \ref{prop:free-refl-quiv-adj} and \ref{prop:2-truncated-nerve-adj} are strikingly similar, and it is tempting to search for a common generalization. The adjoint triangle theorem \cite{Dubuc, Huq} constructs a left adjoint to the first functor in a composable pair whose second functor and composite both admit left adjoints, but this result makes use of a hypothesis that the domain category---$\Cat$ in both of our examples---has reflexive coequalizers, so is inapplicable here. Nevertheless, there may be some useful common abstraction that captures the essential features of the particular construction of the $\Cat$-valued left adjoint and associated quotient functor used in both proofs here.

Finally, the homotopy category functor will play a role in the task of formalizing $\infty${\nobreakdash-}category theory in Lean. Lean's \libmathlib{} contains a popular ``analytic'' definition of $\infty${\nobreakdash-}categories as quasi-categories: a \ldoc[SSet.]{Quasicategory} is a simplicial set satisfying a particular horn filling condition. In a series of PRs that are currently under review, with help from Robin Carlier, Bhavik Mehta, Thomas Murrills, Adam Topaz, Andrew Yang, and Zeyi Zhao, we have formalized a proof that the homotopy category functor preserves finite products. Applying this functor hom-wise converts the simplicially enriched category of quasi-categories to a categorically enriched category of quasi-categories, aka a strict bicategory of quasi-categories. Recent work of Riehl and Verity has shown that the core basic theory of $\infty${\nobreakdash-}categories can be developed within this strict bicategory, whose objects are $\infty${\nobreakdash-}categories, whose morphisms are $\infty${\nobreakdash-}functors, and whose 2-cells are $\infty${\nobreakdash-}natural transformations \cite{RiehlVerity:2022eo}. In particular, the existing \libmathlib{} formalization of adjunctions in a bicategory will specialize to provide the theory of adjunctions between $\infty${\nobreakdash-}categories. In a nearly completed masters thesis project, Jack McKoen is formalizing a proof that the full subcategory $\QCat$ of $\Set^{\DDelta^\op}$ defined by the quasi-categories is cartesian closed, meaning it is enriched over itself. Once this is complete, we will be able to conclude that the strict bicategory of quasi-categories is cartesian closed, allowing us to define limits and colimits in an $\infty$-category and prove that right adjoints preserve limits \cite{InfinityCosmos,RiehlVerity:2022eo}.

\bibliography{cat-has-colimits}

\end{document}